\documentclass[a4paper,12pt,final]{amsart}
\usepackage{times,a4wide,mathrsfs, nameref,showkeys,amssymb,amsmath,amsthm,enumerate,xypic,tikzsymbols,dsfont}

\newcommand{\C}{\mathbb{C}}

\newcommand{\QQ}{\mathbb{Q}}

\newcommand{\PP}{\mathbb{P}}

\DeclareMathOperator{\CH}{CH}

\DeclareMathOperator{\sym}{Sym}

\newtheorem{theorem}{Theorem}[section]

\newtheorem{proposition}[theorem]{Proposition}
\newtheorem{conjecture}[theorem]{Conjecture}
\newtheorem{remark}[theorem]{Remark}

\newtheorem{convention}{Conventions}

\newtheorem{notation}[theorem]{Notation}

\newtheorem{nonumbering}{Theorem}

\newtheorem{nonumberingt}{Acknowledgements}

\begin{document}

\author[Robert Laterveer]
{Robert Laterveer}

\address{Institut de Recherche Math\'ematique Avanc\'ee,
CNRS -- Universit\'e 
de Strasbourg,\
7 Rue Ren\'e Des\-car\-tes, 67084 Strasbourg CEDEX,
FRANCE.}
\email{robert.laterveer@math.unistra.fr}

\title[A strong version of the BV conjecture]{A strong version of the Beauville--Voisin conjecture for certain hyper-K\"ahler fourfolds}

\begin{abstract} A strong version of the Beauville--Voisin conjecture asserts that for hyper-K\"ahler varieties, the subring of the Chow ring generated by divisors, Chern classes and Lagrangian constant cycle subvarieties should inject into cohomology. We verify this in codimension larger than two for Hilbert squares of K3 surfaces, for Fano varieties of lines in cubic fourfolds, and for double EPW sextics.
\end{abstract}

\thanks{This work is supported by ANR grant ANR-20-CE40-0023.}

\maketitle

\section{Introduction}

Given a complex smooth projective variety $X$, let $A^\ast(X)=\oplus_i A^i(X)$ denote the Chow ring with $\QQ$-coefficients.
For hyper-K\"ahler varieties, the Chow ring is expected to have a particularly nice structure. Notably, in the wake of foundational results on the Chow ring of K3 surfaces \cite{BV}, Beauville has formulated the following conjecture:

\begin{conjecture}[Beauville \cite{Beau}]\label{conj1} Let $X$ be a hyper-K\"ahler variety. The $\QQ$-algebra
\[   \langle A^1(X)\rangle\ \ \subset\ A^\ast(X)  \]
generated by divisors injects into cohomology, under the cycle class map.
\end{conjecture}

Conjecture \ref{conj1} is known for Lagrangian fibrations \cite{Rie} and for Hilbert schemes of K3 surfaces \cite{MN}, but is otherwise wide open.

Voisin has suggested the following reinforced version of Beauville's conjecture (this is now commonly known as the ``Beauville--Voisin conjecture''):

\begin{conjecture}[Voisin \cite{V17}]\label{conj2} Let $X$ be a hyper-K\"ahler variety. The $\QQ$-algebra
\[   \langle A^1(X), c_j(T_X)\rangle\ \ \subset\ A^\ast(X)  \]
generated by divisors and Chern classes of the tangent bundle injects into cohomology, under the cycle class map.
\end{conjecture}

Conjecture \ref{conj2} is known for low-dimensional Hilbert schemes of K3 surfaces \cite{V17}, for Fano varieties of lines on cubic fourfolds \cite{V17}, for generalized Kummer varieties \cite{Fu} and for double EPW sextics \cite{L}. Otherwise, this is wide open.

More recently, Voisin has proposed yet a stronger version of the conjecture, involving {\em Lagrangian constant cycle subvarieties\/}. A subvariety $Z\subset X$ is called a constant cycle subvariety if all points of $Z$ represent the same class in the Chow group of zero-cycles of $X$; it is Lagrangian if its dimension equals its codimension (note that this definition implies that the symplectic form of $X$ restricts to 0 on the non-singular locus of $Z$ \cite[Corollary 1.2]{Vcoiso}, whence the "Lagrangian" adjective).

\begin{conjecture}[Voisin \cite{Vcoiso}]\label{conj3} Let $X$ be a hyper-K\"ahler variety. The $\QQ$-algebra
\[   \langle A^1(X), c_j(T_X), Z_k\rangle\ \ \subset\ A^\ast(X)  \]
generated by divisors, Chern classes and Lagrangian constant cycle subvarieties $Z_k\subset X$ injects into cohomology, under the cycle class map.
\end{conjecture}

The goal of this paper is to find some verification of Conjecture \ref{conj3}. The three main results concern three of the most famous families of hyper-K\"ahler fourfolds:

\begin{nonumbering}[=Theorem \ref{main0}] Let $X=S^{[2]}$ be the Hilbert square of a K3 surface $S$. Let
   \[ R^\ast(X):=\bigl\langle  A^1(X), c_j(T_X),  A^2_{(0)}(X), Z_k \bigr\rangle\ \ \subset\ A^\ast(X) \]
  be the graded $\QQ$-algebra generated by divisors, Chern classes, the subgroup $A^2_{(0)}(X)\subset A^2(X)$ and constant cycle surfaces $Z_k\subset X$.
  Then $R^i(X)$ injects into cohomology via the cycle class map for $i\ge 3$.
  
  Moreover, $R^2(X)$ injects into cohomology if and only if $A^2_{(0)}(X)$ injects into cohomology.
\end{nonumbering}

Here, $A^2_{(0)}(X)$ refers to the Shen--Vial splitting of the Chow ring of $X$ \cite{SV}, which is a key ingredient in the proof.

\begin{nonumbering}[=Theorem \ref{main1}] Let $X=F(Y)$ be the Fano variety of lines of a smooth cubic fourfold $Y\subset\PP^5$.
Let
  \[ R^\ast(X):=\bigl\langle  A^1(X), c_j(T_X),  S_k \bigr\rangle\ \ \subset\ A^\ast(X) \]
  be the graded $\QQ$-algebra generated by divisors, Chern classes, and constant cycle surfaces $S_k\subset X$.
  Then $R^i(X)$ injects into cohomology via the cycle class map for $i\ge 3$.
  
  Moreover, $R^2(X)$ injects into cohomology provided $A^2_{(0)}(X)$ injects into cohomology.
\end{nonumbering}

Again, $A^2_{(0)}(X)$ refers to the Shen--Vial splitting of the Chow ring of $X$ \cite{SV}.

\begin{nonumbering}[=Theorem \ref{main2}] Let $X$ be a double EPW sextic with covering involution $\iota$. Let
  \[ R^\ast(X):=\bigl\langle  A^1(X), c_j(T_X), A^2(X)^+, S_k \bigr\rangle\ \ \subset\ A^\ast(X) \]
  be the graded $\QQ$-algebra generated by divisors, Chern classes, $\iota$-invariant codimension 2 classes and constant cycle surfaces $S_k\subset X$.
  Then $R^i(X)$ injects into cohomology via the cycle class map for $i\ge 3$.
  
  Moreover, $R^2(X)$ injects into cohomology if and only if $A^2(X)^+$ injects into cohomology.
\end{nonumbering}

In these three theorems, the $i=4$ part readily reduces to results that were already known (by \cite{V17} resp. \cite{LV}). The
new part of these theorems is thus in codimension $i=3$. The codimension $i=2$ part seems out of reach, at least using the arguments of the present paper (cf. Remarks \ref{hard0}, \ref{hard} and \ref{hard2} below).

(Added in revision: Conjecture \ref{conj3} also holds true in codimension larger than 2 for double EPW quartics; this is proven by Carl Mazzanti \cite[Theorem 7.8 and Corollary 7.9]{Maz}.)

\vskip0.5cm

\begin{convention} In this paper, the word {\sl variety\/} will mean a reduced irreducible scheme of finite type over $\C$. A {\sl subvariety\/} will refer to a (possibly reducible) reduced subscheme which is equidimensional. 

{\bf All Chow groups will be with rational coefficients}: we denote by $A^i(Y)$ the Chow group of codimension $i$ cycles on $Y$ with $\QQ$-coefficients.
The notation $A^i_{hom}(Y)$ will be used to denote the subgroup of homologically trivial cycles.
For a morphism $f\colon X\to Y$, we write $\Gamma_f\in A^\ast(X\times Y)$ for the graph of $f$.

\end{convention}

\vskip0.5cm 

\section{Hilbert squares}

\begin{notation} Let $X=S^{[2]}$ be the Hilbert scheme parametrizing length-2 0-dimensional subschemes of a projective K3 surface $S$. We will write 
  \[ A^\ast(X)= A^\ast_{(\ast)}(X)\]
  for the Shen--Vial decomposition of the Chow ring \cite[Part 2]{SV} (in loc. cit. this is written $\CH^\ast(X)_\ast$).
  It has been established in \cite[Part 2]{SV} that $ A^\ast_{(\ast)}(X)$ is a bigraded ring, under the intersection product (cf. also \cite{NOY} for another proof).
  It is expected that $A^i_{(0)}(X)$ injects into cohomology under the cycle class map; this is currently open for $i=2$.
\end{notation}

Here is the first result of this paper:

\begin{theorem}\label{main0} Let $X=S^{[2]}$ be the Hilbert square of a K3 surface $S$. Let
   \[ R^\ast(X):=\bigl\langle  A^1(X), c_j(T_X),  A^2_{(0)}(X), Z_k \bigr\rangle\ \ \subset\ A^\ast(X) \]
  be the graded $\QQ$-algebra generated by divisors, Chern classes, the subgroup $A^2_{(0)}(X)$ and constant cycle surfaces $Z_k\subset X$.
  Then $R^i(X)$ injects into cohomology via the cycle class map for $i\ge 3$.
  
  Moreover, $R^2(X)$ injects into cohomology provided $A^2_{(0)}(X)$ injects into cohomology.
  \end{theorem}
  
   We actually prove the following more precise result (which is in agreement with the conjectures made by Voisin in \cite{Vcoiso}):

\begin{proposition}\label{CC0} Let $X=S^{[2]}$ be the Hilbert square of a K3 surface. Let $Z\subset X$ be a constant cycle surface. Then
  \[ Z\ \in\ A^2_{(0)}(X)\ .\]
\end{proposition}
  
  This almost immediately implies the theorem: indeed, the proposition, combined with the fact that $A^\ast_{(\ast)}(X)$ is a bigraded ring,
   guarantees that there is an inclusion
  \[ R^3(X)\subset A^3_{(0)}(X)\ .\] 
  But it is known that $A^3_{(0)}(X)$ injects into cohomology via the cycle class map (this is stated in \cite[Introduction]{V6}; it follows from the fact that the projector defining $A^3_{(0)}(X)$ factors
  through a surface and that Murre's Conjecture D is known for surfaces, cf. for instance \cite[Theorem 4.15]{Via}
  or \cite[Lemma 2.20]{23.6}). This settles the $i=3$ part of the theorem. As for the $i=4$ part of the theorem, 
  again by multiplicativity of the bigrading one has 
    \[  A^2_{(0)}(X)\cdot  A^2_{(0)}(X)\ \subset A^4_{(0)}(X)\cong \QQ[c_4(X)]\ \] 
    (Here the last isomorphism follows from \cite[Lemma 13.7(iv)]{SV} combined with the fact that $c_2(X)\in A^2_{(0)}(S)$ which follows from Beauville--Voisin \cite{BV}).
    
   Moreover, it is known that all constant cycle surfaces $Z_k$ represent the same point in $X$ (cf. \cite[Proof of Lemma 3.10]{Vcoiso}), and so for $i=4$ the theorem reduces to the usual Beauville--Voisin conjecture for $X$, which has already been proven in \cite{V17}.

  \begin{proof}(of Proposition \ref{CC0}) Given a constant cycle surface $Z\subset X$, we may decompose its class in $A^2(X)$ with respect to the Shen--Vial decomposition:
  \[ Z=Z_{(0)} + Z_{(2)}\ \ \hbox{in}\ A^2(X)\ ,\]
  where $Z_{(i)}\in A^2_{(i)}(X)$. 
  
  As in \cite[Part 2]{SV}, let $I\subset X\times S$ be the incidence correspondence with projections
    \[ \begin{matrix}  I & \xrightarrow{q} & S\\
         \ \ \   \downarrow{\scriptstyle p} &&\\
           X&&\\
           \end{matrix}\]
    The surface $ S_{\mathfrak o}\subset X$ is defined as 
    \[ S_{\mathfrak o}:= p(q^{-1}{\mathfrak o_S})    \ ,\]
    where ${\mathfrak o}_S   \in S$ is a point representing the distinguished zero-cycle.  
  
 The surface $ S_{\mathfrak o}$ has its class in $A^2_{(0)}(X)$ \cite[Proof of Lemma 15.5]{SV}.
    Using that $A^\ast_{(\ast)}(X)$ is a bigraded ring (and that $A^4_{(0)}(X)\cong \QQ[h^4]$), we can thus write the intersection as
    \[  Z\cdot S_{\mathfrak o} =  Z_{(0)}\cdot S_{\mathfrak o} + Z_{(2)}\cdot S_{\mathfrak o}  
                                               = \alpha h^4 +   Z_{(2)}\cdot S_{\mathfrak o}\ \ \ \hbox{in}\ A^4(X)\ ,
                                               \]
    where $\alpha\in\QQ$ and $Z_{(2)}\cdot S_{\mathfrak o}\in A^4_{(2)}(X)$.
    
    On the other hand, since $Z$ is a constant cycle surface, all points of $Z$ represent the same class in $A^4(X)$, and this class is the same for any constant cycle surface (as mentioned above, this follows from \cite[Proof of Lemma 3.10]{Vcoiso}), which implies
      \[ Z\cdot S_{\mathfrak o} \ \in \QQ[h^4]\ .\]
      It follows that
      \[    Z_{(2)}\cdot S_{\mathfrak o}=0\ \ \hbox{in}\ A^4(X)   \ .\]
 But the map
 \[ \cdot S_{\mathfrak o}\colon\ \ A^2_{(2)}(X)\ \to\      A^4_{(2)}(X)\]
 is known to be an isomorphism \cite[Lemma 15.4]{SV}, and so
   \[ Z_{(2)}=0\ \ \hbox{in}\ A^2(X)\ ,\]
   proving the proposition.
                              \end{proof}

\begin{remark}\label{hard0} The $i=2$ case of Theorem \ref{main0} remains highly challenging. Indeed, the expectation that $A^2_{(0)}(X)$ should inject into cohomology is related to Murre's conjecture D \cite{M} for $S\times S$. Murre's conjecture D is known for products of a surface and a curve \cite{M}, but I am not aware of any results concerning products of surfaces with $p_g\not=0$. 

When $S$ is a Kummer K3 surface, the injectivity of $A^2_{(0)}(X)$ into cohomology would follow from the injectivity of $A^2_{(0)}(A\times A)$ into cohomology, where $A$ is an abelian surface and $A^\ast_{(\ast)}()$ refers to Beauville's splitting for the Chow ring of abelian varieties \cite{B1}. I do not think there are any results concerning the injectivity of $A^2_{(0)}(B)\to H^\ast(B,\QQ)$ for abelian varieties $B$ of dimension larger than 3.
\end{remark}

\section{Fano varieties of lines}

\begin{notation} Let $X=F(Y)$ be the Fano variety of lines of a smooth cubic fourfold $Y\subset\PP^5$. Then $X$ is a hyper-K\"ahler variety of $K3^{[2]}$-type, and these
$X$ form a 20-dimensional locally complete family \cite{BD}.

We will write 
  \[ A^\ast(X)= A^\ast_{(\ast)}(X)\]
  for the Shen--Vial decomposition of the Chow ring \cite[Part 3]{SV} (in loc. cit. this is written $\CH^\ast(X)_\ast$).
  
  It is expected (but not currently known) that $ A^\ast_{(\ast)}(X)$ is a bigraded ring, under the intersection product.
  It is also expected that $A^i_{(0)}(X)$ injects into cohomology under the cycle class map; this is currently open for $i=2$.
\end{notation}

The second result of this paper is the following:

\begin{theorem}\label{main1} Let $X=F(Y)$ be the Fano variety of lines of a smooth cubic fourfold $Y\subset\PP^5$.
Let
  \[ R^\ast(X):=\bigl\langle  A^1(X), c_j(T_X),  S_k \bigr\rangle\ \ \subset\ A^\ast(X) \]
  be the graded $\QQ$-algebra generated by divisors, Chern classes, and constant cycle surfaces $S_k\subset X$.
  Then $R^i(X)$ injects into cohomology via the cycle class map for $i\ge 3$.
  
  Moreover, $R^2(X)$ injects into cohomology provided $A^2_{(0)}(X)$ injects into cohomology.
  \end{theorem}

 We actually prove the following more precise result (which is in agreement with the conjectures made by Voisin in \cite{Vcoiso}):

\begin{proposition}\label{CC} Let $X=F(Y)$ be the Fano variety of lines of a smooth cubic fourfold $Y\subset\PP^5$. Let $S\subset X$ be a constant cycle surface. Then
  \[ S\ \in\ A^2_{(0)}(X)\ .\]
\end{proposition}

This proposition readily implies the theorem: indeed, the proposition, combined with the fact that 
  \[  A^1(X)\cdot A^2_{(0)}(X)\ \subset \ A^3_{(0)}(X)\  \]
  \cite[Proposition A.7]{FLV}, guarantees that there is an inclusion
  \[ R^3(X)\subset A^3_{(0)}(X)\ .\] 
  But it is known that $A^3_{(0)}(X)$ injects into cohomology via the cycle class map (this is because the projector defining $A^3_{(0)}(X)$ factors over a surface and Murre's conjecture D is known for surfaces, cf. \cite[Theorem 4.15]{Via} or the similar argument in
  \cite[Lemma 2.6]{BL}). This settles the $i=3$ part of the theorem. As for the $i=4$ part of the theorem, it is known that all constant cycle surfaces $S_k$ represent the same point in $X$ \cite[Proposition 4.5]{Vcoiso}, and so for $i=4$ the theorem reduces to the usual Beauville--Voisin conjecture for $X$, which has already been proven in \cite{V17}.

\begin{proof}(of Proposition \ref{CC}) 
This is very similar to the proof of Proposition \ref{CC0}. Let us write
  \[ S=S_{0} + S_{2}\ \ \hbox{in}\ A^2(X)\ ,\]
  where $S_{i}\in A^2_{(i)}(X)$.
  We now consider the intersection of the surface $S$ with $h^2$, where $h\in A^1(X)$ denotes the polarization. It is known that $h^2\in A^2(X)$ lies in 
  $A^2_{(0)}(X)$ \cite[Theorem 4.6]{SV}. It is also known that
    \[ \begin{split}  A^2_{(0)}(X)\cdot A^1(X)\cdot A^1(X)\ &\subset\  A^4_{(0)}(X)\ ,\\
    A^2_{(0)}(X)\cdot A^2_{(2)}(X)\ &\subset\ A^4_{(2)}(X)\  \\
    \end{split}\]
    (this follows from \cite[Proposition A.7]{FLV} and \cite[Proposition 22.3]{SV}), and that $A^4_{(0)}\cong\QQ[h^4]$. It follows that we have
    \[  S\cdot h^2 = S_0\cdot h^2 + S_2\cdot h^2 = \alpha h^4 + S_2\cdot h^2\ \ \ \hbox{in}\ A^4_{(0)}(X)\oplus A^4_{(2)}(X)\ ,\]
    where $\alpha\in\QQ$.
    
    As in the proof of Proposition \ref{CC0}, the constant cycle condition now implies that 
      \[ S_2\cdot h^2 =0\ \ \ \hbox{in}\  A^4(X)\ .\]
      Since intersecting with $h^2$ induces an isomorphism
       \[ \cdot h^2\colon\ \ A^2_{(2)}(X)\ \xrightarrow{\cong}\ A^4_{(2)}(X)\ \]
       \cite[Theorems 2.2 and 2.4]{SV}, it follows that $S_2=0$, closing the proof.

    \end{proof}

\begin{remark}
It is possible to give an alternative proof of Proposition \ref{CC}, avoiding the reference to \cite[Proposition A.7]{FLV} but instead using the quadratic relation \cite[Proposition 21.2]{SV}. This alternative proof is similar to the proof of Theorem \ref{main2} below.
\end{remark}

\begin{remark}\label{hard} The $i=2$ case of Conjecture \ref{conj3} for Fano varieties $X$ of lines in cubic fourfolds remains out of reach. Indeed, proving that $A^2_{(0)}(X)$ injects into cohomology seems unfeasible (I do not know of a single cubic fourfold for which this can be done). At least when the cubic does not contain planes, it is known there is an eigenspace decomposition
  \[  A^2(X)= V_{31}\oplus V_{-14}\oplus V_{4}\oplus V_{-2} \ ,\]
  where $V_\tau$ denotes the eigenspace with eigenvalue $\tau$ for the action of $\phi$ on $A^2(X)$
  \cite[Theorem 21.9]{SV}. To settle the $i=2$ case of Conjecture \ref{conj3}, in view of \cite[Lemma 21.12]{SV} it would suffice to show that for any constant cycle surface $S\subset X$, the class $S\in A^2(X)$ avoids $V_4$; unfortunately I do not see any reason why this should be the case.
\end{remark}

\section{Double EPW sextics}

\begin{notation} Double EPW sextics have been constructed by O'Grady \cite{OG2}, \cite{OG4}, \cite{OG5} as double covers $X\to Z$ of certain sextic hypersurfaces $Z\subset\PP^5$.
They form a locally complete family of hyper-K\"ahler varieties of $K3^{[2]}$-type. By construction, they come with a covering involution $\iota$ which is anti-symplectic.
We will write $h\in A^1(X)$ for the polarization, and $c:=c_2(T_X)\in A^2(X)$ for the second Chern class of the tangent bundle.
\end{notation}

Here is the third result of this paper:

\begin{theorem}\label{main2} Let $X$ be a double EPW sextic with covering involution $\iota$. Let
  \[ R^\ast(X):=\bigl\langle  A^1(X), c_j(T_X), A^2(X)^+, S_k \bigr\rangle\ \ \subset\ A^\ast(X) \]
  be the graded $\QQ$-algebra generated by divisors, Chern classes, $\iota$-invariant codimension 2 classes and constant cycle surfaces $S_k\subset X$.
  Then $R^i(X)$ injects into cohomology via the cycle class map for $i\ge 3$.
  
  Moreover, $R^2(X)$ injects into cohomology if and only if $A^2(X)^+$ injects into cohomology.
  \end{theorem}
  
  \begin{proof} We consider the (a priori) smaller $\QQ$-algebra
    \[ P^\ast(X):= \bigl\langle  A^1(X), c_j(T_X), A^2(X)^+ \bigr\rangle\ \ \subset\ A^\ast(X)\ , \]  
    leaving out the constant cycle surfaces. We claim there is equality
    \begin{equation}\label{equal}  P^i(X)= R^i(X)\ \ \hbox{for\ all}\ i\ .\end{equation}
 Since $P^i(X)$ is known to inject into cohomology for $i\ge 3$ \cite[Theorem 3.1]{L}, this claim implies the theorem.   
 
 To prove the claim for $i=4$, we note that there is at least one constant cycle surface present in any double EPW sextic. Indeed, Zhang has shown \cite[Theorem 4.14]{Zh} that the surface of fixed points $S_0\subset X$ is a constant cycle surface. It is known that the class of $S_0$ is $5h^2-{1\over 3} c  $  in $A^2(X)$ \cite[Lemma 4.1]{Fe}. An ingenious argument of Voisin \cite[Proof of Lemma 3.10]{Vcoiso} then implies that any other constant cycle surface $S_k$ must intersect $S_0$, and so all constant cycle surfaces $S_k$ represent the same point. Since $P^4(X)$ is one-dimensional by \cite[Theorem 3.1]{L} (or the earlier \cite[Theorem 1]{LV}), this point must be proportional to $h^4$; this settles the $i=4$ part of the claim.
 
 To prove the $i=2$ (and hence $i=3$) part of the claim \eqref{equal}, let $S=S_k$ be any constant cycle surface in $X$. Let us write
   \[ S = S^+ + S^-\ \ \hbox{in}\ A^2(X)\ ,\]
 where $S^+\in A^2(X)^+$ and $S^-\in A^2(X)^-$ (of course for $S=S_0$ we have $S^-=0$, but for arbitrary $S_k$ this is not clear).
 The map $\cdot h\colon H^2(X,\QQ)^- \to H^4(X,\QQ)^-$ is an isomorphism (indeed, $H^4(X,\QQ)=\sym^2 H^2(X,\QQ)$ and $H^2(X,\QQ)^+=\QQ[h]$), and so there exists a divisor $D\in A^1(X)^-$ such that $S^-=h\cdot D$ in cohomology, i.e.
   \[ S^- -h\cdot D\ \ \in A^2_{hom}(X)^-\ .\]
  
  At this point, we invoke the following equality of correspondences, involving the projector $\Delta_X^-:={1\over 2}(\Delta_X-\Gamma_\iota)$:
  
  \begin{proposition}[\cite{L}]\label{deltaminus}

Let $X$ be a smooth double EPW sextic, and let $\iota$ denote its covering involution. There is a relation

\[
\Delta_X^-= B + \mathcal{J}_1 + \dots + \mathcal{J}_r \ \ \hbox{in}\  A^4(X\times X)\ ,
\]
with the following properties:

\begin{itemize}
\item $B$ is decomposable, i.e.

$$B\in \langle (p_1)^*A^*_{}(X), (p_2)^*A^\ast_{}(X) \rangle ;$$

\item each $\mathcal{J}_i$ is a composition of the correspondences $\Delta_X^-$ and $I\cdot A$, where

\begin{align*}
I & \in  A^2_{}(X\times X),\\
A & \in  \langle (p_i)^*(h^2), (p_i)^*(c) \rangle, \ for\ i\in \{1,2\}&\\
\end{align*}
with $I\cdot A$ occurring at least once in each ${\mathcal J}_i$.

\end{itemize}

    \end{proposition}
    
\begin{proof} This is \cite[Proposition 3.2]{L}. (NB: as observed in \cite[Proposition 2.25]{BL}, the precise expression for $A$ given above is not part of the statement of \cite[Proposition 3.2]{L}, but is immediate from inspection of the proof of \cite[Proposition 3.2]{L}.)
    \end{proof}
 
 Applying the equality of correspondences of Proposition \ref{deltaminus} to the cycle class $ S^- -h\cdot D \in A^2_{hom}(X)^- $, we obtain an equality 
   \begin{equation}\label{act}  S^- - h\cdot D = \bigl( B + \mathcal{J}_1 + \dots + \mathcal{J}_r \bigr){}_\ast (S^- -h\cdot D)\ \ \hbox{in}\ A^2(X)\ .\end{equation}
Let us now check that the right-hand side of \eqref{act} is zero: first, the correspondence $B$, being decomposable, acts as zero on $A^\ast_{hom}(X)$. 

Let us assume the correspondence ${\mathcal J}_i$ contains
a cycle of the form $I\cdot (p_2)^\ast(a)$ where $a\in A^2(X)$. By the projection formula, we find that
  \[  \bigl(    I\cdot (p_2)^\ast(a)  \bigr){}_\ast (\gamma)  =    a\cdot \bigl( I_\ast(\gamma)\bigr)\ \ \hbox{in}\ A^2(X)\ ,\ \ \forall \gamma\in A^2_{hom}(X)\ .\]
But $I_\ast(\gamma)$ is in $A^0_{hom}(X)=0$ and so 
  \[  \bigl(    I\cdot (p_2)^\ast(a)  \bigr){}_\ast (\gamma)  =   0\ \ \hbox{in}\ A^2(X)\ \ \ \forall \gamma\in A^2_{hom}(X)\ ,\]
  hence also 
  \[  ({\mathcal J}_i)_\ast(\gamma)=0 \ \ \forall   \gamma\in A^2_{hom}(X)\ .\]   
  Next, let us assume  ${\mathcal J}_i$ contains
a cycle of the form $I\cdot (p_1)^\ast(D)\cdot (p_2)^\ast(D^\prime)$ where $D, D^\prime\in A^1(X)$. Then (again using the projection formula) we find that
  \[            \bigl(    I\cdot (p_1)^\ast(D)\cdot (p_2)^\ast(D^\prime)  \bigr){}_\ast (\gamma)  =    D^\prime\cdot \bigl( I_\ast(\gamma\cdot D)\bigr)\ \ \hbox{in}\ A^2(X)\ ,\ \ \forall \gamma\in A^2_{hom}(X)\ .\]
  Noting that $ I_\ast(\gamma\cdot D)\in A^1_{hom}(X)=0$ we conclude that
       \[            \bigl(    I\cdot (p_1)^\ast(D)\cdot (p_2)^\ast(D^\prime)  \bigr){}_\ast (\gamma)  = 0  \ \ \forall   \gamma\in A^2_{hom}(X)\ ,\]     
     and so once more we find 
         \[  ({\mathcal J}_i)_\ast(\gamma)=0 \ \ \forall   \gamma\in A^2_{hom}(X)\ .\]   
  The only remaining possibility is now that the correspondence ${\mathcal J}_i$ is built from $\Delta_X^-$ and $I\cdot (p_1)^\ast(a)$ where $a\in A^2(X)$, i.e. we may assume we can write
    \[  {\mathcal J}_i=  {\mathcal J}_i^\prime\circ \bigl(I\cdot (p_1)^\ast(a)\bigr)   \circ \Delta_X^-\ \ \hbox{in}\ A^4(X\times X)\ .\]
    Using the projection formula, we now find
    \[ \begin{split} ({\mathcal J}_i)_\ast   ( S^- -h\cdot D) &=    ({\mathcal J}^\prime_i)_\ast \bigl(I\cdot (p_1)^\ast(a)\bigr){}_\ast (S^- -h\cdot D)\\       
                                                                     & =  ({\mathcal J}^\prime_i)_\ast I_\ast  ( S^-\cdot a - h\cdot D\cdot a)\ .\\
                                                                     \end{split} \]
 The cycle $a$ is a linear combination of $c$ and $h^2$, hence 
   \[  h\cdot D\cdot a\ \in \QQ[h^4] \]
  since we know that $P^4(X)=\QQ[h^4]$. Also,
   \[  S^-\cdot a = (S- S^+)\cdot a \ \in \QQ[h^4]\ ,\]
   since both $S\cdot a$ and $S^+\cdot a$ are in $P^4(X)=\QQ[h^4]$.  
   Since $S^- -h\cdot D$ is homologically trivial, the zero-cycle $S^-\cdot a - h\cdot D\cdot a$ has degree zero, hence is rationally trivial, and so we conclude once again that
     \[ ({\mathcal J}_i)_\ast   ( S^- -h\cdot D)=0      \ \ \hbox{in}\ A^2(X)\ .\]
     In view of the equality \eqref{act}, we thus conclude 
     \[ S^- =h\cdot D\ \ \hbox{in}\ A^2(X)\ .\]
     This proves equality \eqref{equal} for all $i$, and closes the proof of the theorem.
        \end{proof}

  \begin{remark} So far, there is no known ``motivic multiplicative splitting'' of the Chow ring of double EPW sextics, similar to what Shen--Vial have done for Hilbert squares and Fano varieties of lines in cubics \cite{SV}. Hence, the arguments of the two previous sections break down for double EPW sextics. Luckily, the quadratic relation of Proposition \ref{deltaminus} allows to prove a similar result.
  
What we actually prove above is that constant cycle surfaces $S$ in a double EPW sextic $X$ verify
    \[  S \ \ \in\ A^2(X)^+ \oplus A^1(X)^+\cdot A^1(X)^-\ .\]
    Once a satisfactory theory giving a bigrading $A^\ast_{(\ast)}(X)$ has been developed (along the lines of \cite{SV}), it should certainly be the case that
     \[  A^2(X)^+ \oplus A^1(X)^+\cdot A^1(X)^-\ \subset\ A^2_{(0)}(X) \ .\]
  In this sense, the argument for double EPW sextics we give here is analogous to the argument for Fano varieties of lines given in the prior section.   
        \end{remark}    
        
        \begin{remark}\label{hard2} As the sextic $Z\subset\PP^5$ is a (slightly singular) Calabi--Yau variety, the Bloch--Beilinson conjectures would imply that 
          \[ A^2_{hom}(Z)=A^2_{hom}(X)^+=0\ .\]
        This seems very hard to prove; I am not aware of any Calabi--Yau variety for which this has been verified.
          \end{remark}

%

\vskip0.5cm

 \begin{nonumberingt} I am grateful to the referee for pertinent remarks that helped to improve the paper. Thanks to Yotam Ottolenghi for his inspiring books. 
\end{nonumberingt}


\vskip1cm
\begin{thebibliography}{dlPG99}

\bibitem{B1} A. Beauville, Sur l'anneau de Chow d'une vari\'et\'e ab\'elienne, Math. Ann. 273 (1986), 647---651,

\bibitem{Beau} A. Beauville, On the splitting of the Bloch--Beilinson filtration, in: Algebraic cycles and motives (J. Nagel and C. Peters, editors), London Math. Soc. 
Lecture Notes 344, Cambridge University Press 2007,

\bibitem{BD} A. Beauville and R. Donagi, La vari\'et\'e des droites d'une hypersurface cubique de dimension 4, C. R. Acad. Sci., Paris, Sér. I 301 (1985), 703---706,

\bibitem{BV} A. Beauville and C. Voisin, On the Chow ring of a K3 surface, J. Alg. Geom. 13 (2004), 417---426,


\bibitem{BL} M. Bolognesi and R. Laterveer, On the Chow ring of Fano fourfolds of K3 type, in: ``Perspectives on four decades: Algebraic geometry 1980---2020. In memory of Alberto Collino'', Trends in Mathematics, Birkha\"user,

\bibitem{BL2} M. Bolognesi and R. Laterveer, Double EPW sextics and the Voisin filtration on zero-cycles, Math. Z. 310 no. 2 (2025),


%
%
%
%
%
%
%
%
%
%
\bibitem{Fe} A. Ferretti, The Chow ring of double EPW sextics, Algebra Number Theory 6 (2012), no. 3, 539---560,
%

\bibitem{Fu} L. Fu, Beauville--Voisin conjecture for generalized Kummer varieties, Int. Math. Res. Not. 2015, No. 12 (2015), 3878---3898,
		
\bibitem{FLV} L. Fu, R. Laterveer and Ch. Vial, The generalized Franchetta conjecture for some hyper-K\"ahler varieties (with an appendix joint with M.
Shen), Journal Math. Pures et Appliqu\'ees (9) 130 (2019), 1---35,

%
%
%
%
%
%
%
%
%
%
%
%
%

		
%
%
%
%
%
%
%
%
%
\bibitem{23.6} R. Laterveer, Algebraic cycles on some special hyperk\"ahler varieties, Rendiconti di Matematica e delle sue applicazioni 38 no. 2 (2017), 243---276, 

\bibitem{L} R. Laterveer, The Beauville--Voisin conjecture for double EPW sextics, Kyoto Journal of Mathematics 65 no. 4 (2025), 715---736,

\bibitem{LV} R. Laterveer and Ch.Vial, Zero-cycles on double EPW sextics, Commun. Contemp. Math. 23 no. 4 (2021), 
%

\bibitem{MN} D. Maulik and A. Negut, Lehn's formula in Chow and conjectures of Beauville and Voisin, J. Inst. Math. Jussieu 21 no. 3 (2022), 933---971,

\bibitem{Maz} C. Mazzanti, On the Chow ring of double EPW quartics, arXiv:2603.02251,


\bibitem{M} J. Murre, On a conjectural filtration on the Chow groups of an algebraic variety, parts I and II, Indag. Math. 4 (1993), 177---201,

\bibitem{NOY} A. Negut, G. Oberdieck and Q. Yin, Motivic decompositions for the Hilbert scheme of points of a K3 surface, J. Reine Angew. Math. 778 (2021), 65---95,
%
%
\bibitem{OG2} K. O'Grady, Irreducible symplectic $4$-folds and Eisenbud--Popescu--Walter sextics, Duke Math. J. {134} no. 1 (2006), 99---137,
%
%
\bibitem{OG4} K. O'Grady, EPW-sextics: taxonomy, Manuscripta Math. {138} no. 1 (2012), 221--272,
		
\bibitem{OG5} K. O'Grady, Double covers of EPW-sextics, Michigan Math. J. {62} (2013), 143--184,
%
%

\bibitem{Rie} U. Rie{\ss}, On Beauville's conjectural weak splitting property, Int. Math. Res. Not. 2016, No. 20, 6133---6150, 
%
%
%
\bibitem{SV} M. Shen and Ch. Vial, The Fourier transform for certain hyperK\"ahler fourfolds, Memoirs of the AMS 240 (2016),
%
%

\bibitem{Via} Ch. Vial, Niveau and coniveau filtrations on cohomology groups and Chow groups, Proceedings London Math. Soc. 106 no. 2 (2013), 410---444,

\bibitem{V6} Ch. Vial, On the motive of some hyperk\"ahler varieties, J. Reine Angew. Math. 725 (2017), 235---247,

%
%
%

\bibitem{V0} C. Voisin, Intrinsic pseudo-volume forms and K-correspondences, in: The Fano conference (Eds. A. Collino et alii), Univ. Torino, Torino 2004,


\bibitem{V17} C.~Voisin, On the Chow ring of certain algebraic hyperk\"ahler manifolds, Pure Appl. Math. Q. 4 no. 3 part 2 (2008), 
613---649,

\bibitem{Vo} C. Voisin, Chow Rings, Decomposition of the Diagonal, and the Topology of Families, Princeton University Press, Princeton and Oxford, 2014,

\bibitem{Vcoiso} C. Voisin, Remarks and questions on coisotropic subvarieties and 0-cycles of hyper-Kähler varieties, in: K3 surfaces and their moduli, 365---399,
Progr. Math. 315, Birkhäuser 2016,

%
%
%

\bibitem{Zh} R. Zhang, One-cycles on Gushel--Mukai fourfolds and the Beauville--Voisin filtration, Science China Mathematics 67 (2024), 713---732,




\end{thebibliography}
\end{document}